\documentclass[12pt]{article}
\usepackage{amsmath,amssymb,amsthm,booktabs}
\usepackage{mathrsfs}
\usepackage{geometry}
\usepackage{pifont}
\usepackage[OT1]{fontenc}
\usepackage[utf8]{inputenc}
\usepackage[colorlinks,citecolor=blue,urlcolor=blue]{hyperref}
\usepackage{txfonts}
\usepackage{indentfirst}
\usepackage{multirow}
\usepackage{float,subfig}

\usepackage{bm}
\usepackage{euscript}
\usepackage{graphicx}
\usepackage{multicol}
\usepackage[usenames,dvipsnames,svgnames,table]{xcolor}

\usepackage[round]{natbib}
\numberwithin{equation}{section} \theoremstyle{plain}
\newtheorem{theorem}{Theorem}[section]

\newtheorem{assumption}{Assumption}[section]

\newtheorem{remark}{Remark}[section]

\begin{document}

\newcommand{\gai}[1]{{#1}}


\makeatletter
\def\ps@pprintTitle{%
  \let\@oddhead\@empty
  \let\@evenhead\@empty
  \let\@oddfoot\@empty
  \let\@evenfoot\@oddfoot
}
\makeatother

\newcommand\tabfig[1]{\vskip5mm \centerline{\textsc{Insert #1 around here}}  \vskip5mm}

\vskip2cm

\title{Parameter learning: stochastic optimal control approach with reinforcement learning}
\author{Shuzhen Yang\thanks{Shandong University-Zhong Tai Securities Institute for Financial Studies, Shandong University, PR China, (yangsz@sdu.edu.cn). This work was supported by the National Key R\&D program of China (Grant No.2018YFA0703900, ZR2019ZD41), National Natural Science Foundation of China (Grant No.11701330), and Taishan Scholar Talent Project Youth Project.}
}
\date{}
\maketitle

\begin{abstract}
In this study, we develop a stochastic optimal control approach with reinforcement learning structure to learn the unknown parameters appeared in the drift and diffusion terms of the stochastic differential equation. By choosing an appropriate cost functional, based on a classical optimal feedback control, we translate the original optimal control problem to a new control problem which takes place the unknown parameter as control, and the related optimal control can be used to estimate the unknown parameter. We establish the mathematical framework of the dynamic equation for the exploratory state, which is consistent with the existing results. Furthermore, we consider the linear stochastic differential equation case where the drift or diffusion term with unknown parameter. Then, we investigate the general case where both the drift and diffusion terms contain unknown parameters. For the above cases, we show that the optimal density function satisfies a Gaussian distribution which can be used to estimate the unknown parameter. Based on the obtained estimation of the parameters, we can do empirical analysis for a given model.
\end{abstract}

\noindent KEYWORDS: Parameters estimation; Stochastic optimal control; Reinforcement learning; Linear SDE; Exploratory

\section{Introduction}

\cite{Z23} commented that "We strive to seek optimality, but often find ourselves trapped in bad optimal, solutions that are either local optimizers, or are too rigid to leave any room for errors". Indeed, exploration through randomization brings a way to break this curse of optimality. In this present study, we try to balance the model based and model free methods, that is we aim to use the reinforcement learning structure to explore the unknown parameters appeared in the model based method.

Recently, \cite{WZZ20a} first considered the reinforcement learning in continuous time and spaces, and an exploratory formulation for the dynamics equation of the state and an entropy-regularized reward function were introduced. Then, \cite{WZZ20b} solved the continuous time mean variance portfolio selection problem under the reinforcement learning stochastic optimal control framework. \cite{TZZ22} considered the exploratory Hamilton-Jacobi-Bellman (HJB) equation formulated by \cite{WZZ20a} in the context of reinforcement learning in continuous time, and established the well-posedness and regularity of the viscosity solution to the HJB equation. \cite{GXZ22} considered the temperature control problem for Langevin diffusions in the context of nonconvex optimization under the regularized exploratory formulation developed by \cite{WZZ20a}. Differing from the continuous-time entropy-regularized reinforcement learning problem of \cite{WZZ20a}, \cite{HWZ23} proposed a Choquet regularizer to measure and manage the level of the exploration for reinforcement learning.

For the classical stochastic optimal control theory and related applications, we refer the readers to monographs \cite{FR75,YZ99,FM06}. For the recursive stochastic optimal control problem and related dynamic programming principle, see \cite{PP90,P90,P92}. Well-known that the classical optimal control theory is applied to solve the mean-variance model, see \cite{ZL00,BC10,BMZ14,DJKX20}. When we consider the applications of the classical optimal control problem, for example in the mean-variance investment portfolio problem, we first need to estimate the parameters appeared in the model. Some statistic methods could be employed such as moment estimation, maximum likelihood estimation which heavily depend on the stringent assumptions on the observed samples. Furthermore, based on the observations (historical data), it is difficult to estimate the time varying parameter appeared in the model based problems. The reinforcement learning algorithm has been widely used to study optimization, engineering and finance, etc. In particular, \cite{WZZ20a} first introduced the reinforcement learning in continuous time stochastic optimal control problem.

However, it is important to find a better estimation for the parameter appeared in the model. Based on the estimation of the parameter, we can do empirical and time trend analysis. Therefore, in this study, we first show how to use a stochastic optimal control approach with reinforcement learning structure to learn the unknown parameters appeared in the dynamic equation. We consider the following stochastic optimal control problem with unknown deterministic parameter $\beta(\cdot)$,
\begin{equation}\label{in-eq-sde}
\mathrm{d}X^u_s=b(\beta_s,X_s^u,u_s)\mathrm{d}s+\sigma(\beta_s,X_s^u,u_s)\mathrm{d}W_s,\ X^u_t=x,
\end{equation}
where $b,\sigma$ are given deterministic function, $W(\cdot)$ is a Brownian motion, and $u(\cdot)$ is an input control. In the input control $u(\cdot)$, we can observe the output state $X^u(\cdot)$. In the classical mean-variance investment portfolio framework, the parameter $\beta(\cdot)$ in (\ref{in-eq-sde}) can be used to describe the mean and volatility of risky asset. Thus, it is useful to obtain the estimation of the unknown parameter $\beta(\cdot)$.

Now, we show the details of the stochastic optimal control approach with reinforcement learning structure for learning the unknown parameters. \textbf{Step 1:} Based on equation (\ref{in-eq-sde}), by choosing an appropriate cost functional, we obtain a feedback optimal control $u^*(\cdot)$ which contains the unknown parameter $\beta(\cdot)$, denote by $u^*_s=u^*(\beta_s,s,X^*_s),\ t\leq s\leq T$. \textbf{Step 2:} We take place the unknown parameter $\beta(\cdot)$ in the optimal control $u^*(\cdot)$ as a new deterministic control $\rho(\cdot)$, and the input control is denoted by $u^{\rho}_s=u^{\rho}(\rho_s,s,X^{\rho}_s),\ t\leq s\leq T$. \textbf{Step 3:} We put the new control $u^{\rho}(\cdot)$ into equation (\ref{in-eq-sde}) and the related cost functional, and establish a new optimal control problem. Indeed, the optimal control of the new exploratory control problem can be used to estimate the unknown parameters. Therefore, we consider the linear stochastic differential equation case where the drift or diffusion term contains the unknown parameter, and the general case where both the drift and diffusion terms contain unknown parameters. For the above cases, we find that the optimal density function can be used to learn the unknown parameters.

In this present paper, we focus on develop the theory of a stochastic optimal control approach with reinforcement learning structure to learn the unknown parameters appeared in the equation (\ref{in-eq-sde}). For the application of this new theory, we leave it as the future work. However, we refer the reader to the following references for the study of the algorithm of the optimal control. A unified framework to study policy evaluation and the associated temporal difference methods in continuous time  was investigated in \cite{JZ22a}.  \cite{JZ22b} studied the policy gradient for reinforcement learning in continuous time and space under the regularized exploratory formulation developed by \cite{WZZ20a}.  \cite{D22} studied the optimal stopping problem under the exploratory framework, and established the related HJB equation and algorithm. Furthermore, \cite{JZ22c} introduced a q-learning in continuous time and space. For the policy evaluation, we follow \cite{D00} for learning the value function. An introduction of reinforcement learning could be found in \cite{SB18}, and the deep learning and related topic see \cite{GBC16}.

The main contributions of this study are threefold:

(i). To learn the unknown parameters appeared in the dynamic equation of the state, we first develop a stochastic optimal control approach with reinforcement learning structure.

(ii). By choosing an appropriate cost functional, based on a classical optimal feedback control, we translate the original optimal control problem into a new control problem, in which the related optimal control is used to estimate unknown parameter. Based on the obtained estimations of the parameters, we can do empirical analysis for a given model.

(iii). We consider the linear stochastic differential equation case where the drift or diffusion term contains unknown parameter, and the general case where both the drift and diffusion terms contain unknown parameters. We show that the optimal density function satisfies a Gaussian distribution which can be used to learn the unknown parameters.

The remainder of this study is organized as follows. In section \ref{sec:2}, we formulate the stochastic optimal control approach with reinforcement learning structure, and show how to learn the unknown parameter. Then, we investigate the exploratory HJB equation for the stochastic optimal control approach in Section \ref{sec:3}. In Section \ref{sec:4}, we consider the problem of the Linear SDE case, where the drift or diffusion term is allowed to contain the unknown parameter. Furthermore, in Section \ref{sec:5}, we general the model in Section \ref{sec:4} to the case where both the drift and diffusion terms contain the unknown parameters. Then, we give an example to verify the main results. Finally, we conclude this study in Section \ref{sec:6}.

\section{Formulation of the model}\label{sec:2}

Given a probability space $(\Omega, \mathcal{F},P)$, Brownian motion $\{W_t\}_{t\geq 0}$, and filtration $\{\mathcal{F}_t\}_{t\geq 0}$ generated by $\{W_t\}_{t\geq 0}$, where $\mathcal{F}=\mathcal{F}_T$ with a given terminal time $T$. We introduce the following stochastic differential equation with the deterministic unknown parameter $\beta_s\in \mathbb{R},\ t\leq s\leq T$, and control $u(\cdot)\in \mathcal{U}[t,T]$,
\begin{equation}\label{eq-sde}
\mathrm{d}X^u_s=b(\beta_s,X_s^u,u_s)\mathrm{d}s+\sigma(\beta_s,X_s^u,u_s)\mathrm{d}W_s,\ X^u_t=x,
\end{equation}
where $\mathcal{U}[t,T]$ is the set of all progressing measurable, square integrate processes on $[t,T]$, and the processes take value in a Euclidean space. Based on the state $\{X^u_s\}_{t\leq s\leq T}$, we consider the running and terminal cost functional,
\begin{equation}\label{cost-1}
J(t,x;u(\cdot))=\mathbb{E}\left[\int_t^Tf(X_s^u,u_s)\mathrm{d}s+\Phi(X_T^u)\right],
\end{equation}
and the target is to minimize $J(t,x;u(\cdot))$ over $u(\cdot)\in \mathcal{U}[t,T]$, that is
$$
\inf_{u(\cdot)\in \mathcal{U}[t,T]}J(t,x;u(\cdot)).
$$
The conditions on functions $b,\sigma,f$ and $\Phi$ in equations (\ref{eq-sde}) and (\ref{cost-1}) will be given later. Under mild conditions, we can show that there exists a feedback optimal control $u^*_s=u^*(\beta_s,s,X^*_s),\ t\leq s\leq T$ (see \cite{YZ99} for further details), where $X^*$ is the optimal state with the optimal control $u^*$.
\begin{remark}
Indeed, we can choose a cost functional $J(t,x;u(\cdot))$ such that there exists a feedback optimal control which contains the unknown parameter $\beta(\cdot)$. See Sections \ref{sec:4} and \ref{sec:5} for further details.
\end{remark}

Now, we give the assumptions of this study. The functions $b,\sigma,f$ and $\Phi$ are continuous on all the variables.
\begin{assumption}\label{ass-1}
The functions $b$ and $\sigma$ are linear growth and Lipschaitz continuous on the second and third variables with Lipschitz constant $L$.
\end{assumption}
\begin{assumption}\label{ass-2}
The functions $f(x,u)$ and $\Phi(x)$ are polynomial growth on $x$ and $u$.
\end{assumption}

\begin{assumption}\label{ass-3}
$\beta(\cdot)$ is a deterministic piecewise continuous function on $[t,T]$. The feedback optimal control $u^*(\beta_s,s,x),\ t\leq s\leq T$ is uniformly Lipschitz continuous on $x$, with Lipschitz constant $L$.
\end{assumption}

In practical analysis, we always assume the form of the functions $b,\sigma,f$ and $\Phi$ and leave room for estimating the unknown parameter $\beta(\cdot)$. Some classical statistics estimations are implemented, for example moment estimation, maximum likelihood estimation, etc. However, these estimation methods are heavily based on the observations. In this study, we aim to develop a new estimations method based on the explicit feedback control $u^*(\cdot)$.

It should be note that we can observe the value of the state $X^u(\cdot)$ with a given control $u(\cdot)$, but not the unknown parameter $\beta(\cdot)$. In the following, we develop the model used to estimate the parameter $\beta(\cdot)$. Since we don't know the value of the parameter $\beta(\cdot)$, we take place $\beta(\cdot)$ as a new deterministic control process $\rho(\cdot)$ in the feedback optimal control $u^*(\beta_s,s,X^*_s)$, and denotes as $u^{\rho}_s=u^{*}(\rho_s,s,X^{\rho}_s)$, where $X^{\rho}_s$ satisfies the following stochastic differential equation,
\begin{equation}\label{eq-sde-1}
\mathrm{d}X^{\rho}_s=\hat{b}(\beta_s,X_s^{\rho},\rho_s)\mathrm{d}s
+\hat{\sigma}(\beta_s,X_s^{\rho},{\rho}_s)\mathrm{d}W_s,\ X^{\rho}_t=x,
\end{equation}
where $\hat{b}(\beta_s,X_s^{\rho},\rho_s)={b}(\beta_s,X_s^{\rho},u^{\rho}_s)$ and $\hat{\sigma}(\beta_s,X_s^{\rho},\rho_s)={\sigma}(\beta_s,X_s^{\rho},u^{\rho}_s)$.
Based on the above given assumptions, we have the following unique results for the solution of equation (\ref{eq-sde-1}).
\begin{theorem}
Let Assumptions \ref{ass-1} and \ref{ass-3} hold. The equation (\ref{eq-sde-1}) admits a unique solution.
\end{theorem}

Furthermore, we rewrite the cost functional (\ref{cost-1}) as follows,
\begin{equation}\label{cost-2}
\hat{J}(t,x;\rho(\cdot))=\mathbb{E}\left[\int_t^T\hat{f}(X_s^{\rho},\rho_s)\mathrm{d}s+\hat{\Phi}(X_T^{\rho})\right],
\end{equation}
where $\rho_s\in \mathbb{R}$, $\hat{f}(X_s^{\rho},\rho_s)=f(X_s^{\rho},u^{\rho}_s)$, and $\hat{\Phi}(X_T^{\rho})={\Phi}(X_T^{\rho})$, $t\leq s\leq T$. We denote $\mathcal{D}[t,T]$=\{all deterministic piecewise continuous functions on $[t,T]$\}. Obviously, we have the following equivalence between the cost functional $J(t,x;u(\cdot))$ and $\hat{J}(t,x;\rho(\cdot))$.
\begin{theorem}\label{the-1}
Let Assumption \ref{ass-3} hold, $J(t,x;u(\cdot))$ be the cost functional given in (\ref{cost-1}), and $J(t,x;\rho(\cdot))$ be given in (\ref{cost-2}). Then, we have
$$
\inf_{\rho(\cdot)\in\mathcal{D}[t,T]}\hat{J}(t,x;\rho(\cdot))=\inf_{u(\cdot)\in \mathcal{U}[t,T]}J(t,x;u(\cdot)).
$$
\end{theorem}
\begin{proof}
Note that, $\rho(\cdot)\in\mathcal{D}[t,T]$ is a deterministic piecewise continuous process. From Assumption \ref{ass-3}, it is easily to verify that $u^{\rho}(\cdot)=u^{*}(\rho(\cdot),\cdot,X^{\rho}(\cdot))\in \mathcal{U}[t,T]$. Thus,
$$
\inf_{\rho(\cdot)\in\mathcal{D}[t,T]}\hat{J}(t,x;\rho(\cdot))\geq \inf_{u(\cdot)\in \mathcal{U}[t,T]}J(t,x;u(\cdot)).
$$

On the contrary, since $u^*(\cdot)=u^{*}(\beta(\cdot),\cdot,X^{\rho}(\cdot))\in \mathcal{D}[t,T]$ and
$$
J(t,x;u^*(\cdot))=\inf_{u(\cdot)\in \mathcal{U}[t,T]}J(t,x;u(\cdot)),
$$
thus, we have
$$
\inf_{\rho(\cdot)\in\mathcal{D}[t,T]}\hat{J}(t,x;\rho(\cdot))\leq \inf_{u(\cdot)\in \mathcal{U}[t,T]}J(t,x;u(\cdot)).
$$
This completes the proof.
\end{proof}
\begin{remark}
Based on Theorem \ref{the-1}, if the cost functional $\hat{J}(t,x;\rho(\cdot))$ admits a unique optimal control $\rho^*(\cdot)$, we have that $\beta(\cdot)=\rho^*(\cdot)$, which means that we can use the above new optimal control problem to estimate the parameter in the original optimal control problem. Indeed, we can choose a cost functional which satisfies the above inference. Further details can be found in Sections \ref{sec:4} and \ref{sec:5}.
\end{remark}

From equation (\ref{eq-sde-1}), we can observe the value of the state $X^{\rho}(\cdot)$ based on an input control $\rho(\cdot)$. Following the idea investigated in \cite{WZZ20a}, we use reinforcement learning (RL) structure to learn the optimal control $\rho^*(\cdot)$. Motivated with the law of large numbers of samples of state under the density function $\pi(\cdot)$, \cite{WZZ20a} introduced the following exploratory dynamic equation for state $X^{\pi}(\cdot)$,
\begin{equation}\label{eq-sde-2}
\mathrm{d}X^{\pi}_s=\tilde{b}(\beta_s,X_s^{\pi},\pi_s)\mathrm{d}s
+\tilde{\sigma}(\beta_s,X_s^{\pi},{\pi}_s)\mathrm{d}W_s,\ X^{\pi}_t=x,
\end{equation}
where $\pi_s$ is the density function of the input control $\rho_s$,
$$
\tilde{b}(\beta_s,x,\pi_s)=\int_{\mathbb{R}}\hat{b}(\beta_s,x,\rho)\pi_s(\rho)\mathrm{d}\rho,
$$
and
$$
\tilde{\sigma}(\beta_s,x,{\pi}_s)=
\sqrt{\int_{\mathbb{R}}\hat{\sigma}^2(\beta_s,x,\rho)\pi_s(\rho)\mathrm{d}\rho},
$$
for $t\leq s\leq T$.

It should be note that, the observations of the exploratory state $X^{\pi}(\cdot)$ should satisfy the equation (\ref{eq-sde-1}). Based on this fact, we can prove that $X^{\pi}(\cdot)$ satisfies equation (\ref{eq-sde-2}).
\begin{theorem}\label{the-rep}
Let Assumptions \ref{ass-1}, \ref{ass-2}, and \ref{ass-3} hold, and the state $X^{\pi}(\cdot)$ satisfies a diffusion process. Then, $X^{\pi}(\cdot)$ satisfies equation (\ref{eq-sde-2}).
\end{theorem}
\begin{proof}
We assume that the state $X^{\pi}(\cdot)$ satisfies the following diffusion process,
\begin{equation}\label{eq-sde-20}
\mathrm{d}X^{\pi}_s=h(\beta_s,X_s^{\pi},\pi_s)\mathrm{d}s
+g(\beta_s,X_s^{\pi},{\pi}_s)\mathrm{d}W_s,\ X^{\pi}_t=x,
\end{equation}
where the functions $h$ and $g$ are needed to be determined.

Let us consider a sequence of samples $\{x^i(\cdot)\}_{i=1}^N$ which are the observations of $X^{\pi}(\cdot)$ with $\{W^i(\cdot)\}_{i=1}^N$ and $\{\rho^i(\cdot)\}_{i=1}^N$, where $W^i(\cdot)$ is the independent sample paths of the Brownian motion $W(\cdot)$, and $\rho^i(\cdot)$ is the control which obeys the density function $\pi(\cdot)$, $1\leq i\leq N$. For given $t\leq s\leq T$ and sufficiently small $\Delta s>0$, $x^i(\cdot)$ satisfies
\begin{equation}\label{eq-sde-3}
x^{i}_{s+\Delta s}-x^i_s=\int_s^{s+\Delta s}\hat{b}(\beta_r,x_r^{i},\rho^i_r)\mathrm{d}r
+\int_s^{s+\Delta s}\hat{\sigma}(\beta_r,x^i_r,{\rho}^i_r)\mathrm{d}W^i_r,\ x^{i}_t=x,
\end{equation}
which leads to
\begin{align*}
x^{i}_{s+\Delta s}-x^i_s=&\hat{b}(\beta_s,x_s^{i},\rho^i_s)\Delta s+\int_s^{s+\Delta s}\hat{\sigma}(\beta_r,x^i_r,{\rho}^i_r)\mathrm{d}W^i_r\\
&+\int_s^{s+\Delta s}\left[\hat{b}(\beta_r,x_r^{i},\rho^i_r)-\hat{b}(\beta_s,x_s^{i},\rho^i_s)\right]\mathrm{d}r.
\end{align*}
Applying the classical law of larger numbers, we have the following convergence results in probability:
\begin{align*}
&\lim_{N\to\infty}\frac{1}{N}\left(x^{i}_{s+\Delta s}-x^i_s\right)=\mathbb{E}[X^{\pi}_{s+
\Delta s}-X^{\pi}_s];\\
&\lim_{N\to\infty}\frac{1}{N}\hat{b}(\beta_s,x_s^{i},\rho^i_s)\Delta s=\mathbb{E}\left[\int_{\mathbb{R}}\hat{b}(\beta_s,X^{\pi}_s,\rho)\pi_s(\rho)\mathrm{d}\rho\Delta s\right];\\
&\lim_{N\to\infty}\frac{1}{N}\int_s^{s+\Delta s}\hat{\sigma}(\beta_r,x^i_r,{\rho}^i_r)\mathrm{d}W^i_r=0;\\
&\lim_{N\to\infty}\frac{1}{N}\int_s^{s+\Delta s}\left[\hat{b}(\beta_r,x_r^{i},\rho^i_r)-\hat{b}(\beta_s,x_s^{i},\rho^i_s)\right]\mathrm{d}r=o(\Delta s),
\end{align*}
where the last equality is obtained by the continuous dependence of stochastic differential equation, $o(\Delta s)$ is the higher order infinitesimal of $\Delta s$, and thus,
\begin{equation}\label{the-rep-equa-1}
\mathbb{E}[X^{\pi}_{s+
\Delta s}-X^{\pi}_s]=\mathbb{E}\left[\int_{\mathbb{R}}\hat{b}(\beta_s,X^{\pi}_s,\rho)\pi_s(\rho)\mathrm{d}\rho\Delta s\right]+o(\Delta s).
\end{equation}
Dividing on both sides of equation (\ref{the-rep-equa-1}) by $\Delta s$, and letting $\Delta s\to 0$, it follows that
\begin{equation}\label{the-rep-equa-2}
\lim_{\Delta s\to 0}\frac{\mathbb{E}[X^{\pi}_{s+
\Delta s}-X^{\pi}_s]}{\Delta s}=\mathbb{E}\left[\int_{\mathbb{R}}\hat{b}(\beta_s,X^{\pi}_s,\rho)\pi_s(\rho)\mathrm{d}\rho\right].
\end{equation}
Combining equations (\ref{eq-sde-20}) and (\ref{the-rep-equa-2}), we have that
\begin{equation}\label{the-rep-equa-3}
\mathbb{E}\left[h(\beta_s,X^{\pi}_s,\pi_s)\right]
=\mathbb{E}\left[\int_{\mathbb{R}}\hat{b}(\beta_s,X^{\pi}_s,\rho)\pi_s(\rho)\mathrm{d}\rho\right].
\end{equation}

It is worth noting that, for any given $A\in \mathcal{F}_s$, from equations (\ref{eq-sde-20}) and (\ref{eq-sde-3}), we have
$$
\mathrm{d}X^{\pi}_s1_{A}=h(\beta_s,X_s^{\pi}1_A,\pi_s)1_A\mathrm{d}s
+g(\beta_s,X_s^{\pi}1_A,{\pi}_s)1_A\mathrm{d}W_s,
$$
and
$$
x^{i}_{s+\Delta s}1_{A^i}-x^i_s1_{A^i}=\int_s^{s+\Delta s}\hat{b}(\beta_r,x_r^{i}1_{A^i},\rho^i_r)1_{A^i}\mathrm{d}r
+\int_s^{s+\Delta s}\hat{\sigma}(\beta_r,x^i_r1_{A^i},{\rho}^i_r)1_{A^i}\mathrm{d}W^i_r,
$$
where $1_A$ is the indicator function on set $A$, and $A^i$ is measurable about the filtration generated by $W^i_r,\ t\leq r\leq s$. Then, similar with the above proof and equation (\ref{the-rep-equa-3}), we have
\begin{equation}\label{the-rep-equa-4}
\mathbb{E}\left[h(\beta_s,X^{\pi}_s,\pi_s)1_A\right]
=\mathbb{E}\left[\int_{\mathbb{R}}\hat{b}(\beta_s,X^{\pi}_s,\rho)\pi_s(\rho)\mathrm{d}\rho1_A\right],
\end{equation}
and thus
\begin{equation}\label{the-rep-equa-5}
h(\beta_s,X^{\pi}_s,\pi_s)=\int_{\mathbb{R}}\hat{b}(\beta_s,X^{\pi}_s,\rho)\pi_s(\rho)\mathrm{d}\rho,\quad P-a.s.
\end{equation}

Applying It\^{o} formula to $(X^{\pi}_s)^2$, it follows that
\begin{equation}\label{eq-sde-21}
\mathrm{d}(X^{\pi}_s)^2=\left[2X^{\pi}_sh(\beta_s,X_s^{\pi},\pi_s)+g^2(\beta_s,X_s^{\pi},{\pi}_s)\right]\mathrm{d}s
+2X^{\pi}_sg(\beta_s,X_s^{\pi},{\pi}_s)\mathrm{d}W_s.
\end{equation}
Now, similar with the proof in equation (\ref{the-rep-equa-5}), we consider the samples $\{(x^i(\cdot))^2\}_{i=1}^N$, and obtain that
\begin{equation}\label{the-rep-equa-6}
2X^{\pi}_sh(\beta_s,X_s^{\pi},\pi_s)+g^2(\beta_s,X_s^{\pi},{\pi}_s)
=\int_{\mathbb{R}}\left[2X^{\pi}_s\hat{b}(\beta_s,X_s^{\pi},\rho)
+\hat{\sigma}^2(\beta_s,X_s^{\pi},\rho)\right]\pi_s(\rho)\mathrm{d}\rho.
\end{equation}
Combining equations (\ref{the-rep-equa-5}) and (\ref{the-rep-equa-6}), we have
\begin{equation}\label{the-rep-equa-7}
g^2(\beta_s,X_s^{\pi},{\pi}_s)
=\int_{\mathbb{R}}\hat{\sigma}^2(\beta_s,X_s^{\pi},\rho)\pi_s(\rho)\mathrm{d}\rho,\quad P-a.s.
\end{equation}
Thus, we can choose
$$
g(\beta_s,X_s^{\pi},{\pi}_s)
=\sqrt{\int_{\mathbb{R}}\hat{\sigma}^2(\beta_s,X_s^{\pi},\rho)\pi_s(\rho)\mathrm{d}\rho},
$$
which completes the proof.
\end{proof}

\begin{remark}\label{remk-1}
In Theorem \ref{the-rep}, we show that the state $X^{\pi}(\cdot)$ satisfies equation (\ref{eq-sde-2}). Furthermore, from equation (\ref{the-rep-equa-7}), for $t\leq s\leq T$, we have that
$$
g(\beta_s,X_s^{\pi},{\pi}_s)=\sqrt{\int_{\mathbb{R}}\hat{\sigma}^2(\beta_s,X_s^{\pi},\rho)\pi_s(\rho)\mathrm{d}\rho},
$$
or
$$
g(\beta_s,X_s^{\pi},{\pi}_s)=-\sqrt{\int_{\mathbb{R}}\hat{\sigma}^2(\beta_s,X_s^{\pi},\rho)\pi_s(\rho)\mathrm{d}\rho}.
$$
Indeed, the Brownian motion $W_s$ satisfies the normal distribution with mean $0$ and variance $s$. Thus, we just need to choose one of them.
\end{remark}

Furthermore, in the framework of exploratory, we rewrite the cost functional (\ref{cost-2}) as follows,
\begin{equation}\label{cost-3}
\tilde{J}(t,x;\pi(\cdot))=
\mathbb{E}\left[\int_t^T\int_{\mathbb{R}}\left[\hat{f}(X_s^{\pi},\rho)\pi_s(\rho)
+\lambda\pi_s(\rho)\ln\pi_s(\rho)\right]\mathrm{d}
\rho\mathrm{d}s+\hat{\Phi}(X_T^{\pi})\right],
\end{equation}
where the term
$$
\lambda\int_{\mathbb{R}}\pi_s(\rho)\ln\pi_s(\rho)\mathrm{d}
\rho
$$
denotes the Shanon's differential entropy used to measure the level of exploratory, and $\lambda>0$ is the temperature constant balanced the exploitation and exploration. We denote all the admissible density functions control on $\mathcal{D}[t,T]$ as $\mathcal{A}[t,T]$. Thus, our target is to minimum the cost functional (\ref{cost-3}) over $\mathcal{A}[t,T]$.

\section{Hamilton-Jacobi-Bellman Approach}\label{sec:3}
Now, we consider equation (\ref{eq-sde-2}) with cost functional (\ref{cost-3}). Similar with the manner in \cite{WZZ20a}, employing the classical dynamic programming principle, one obtains that
$$
v(t,x)=\inf_{\pi(\cdot)\in \mathcal{A}[t,s]}\mathbb{E}\left[\int_t^s\int_{\mathbb{R}}\left[\hat{f}(X_r^{\pi},\rho)\pi_r(\rho)
+\lambda\pi_r(\rho)\ln\pi_r(\rho)\right]\mathrm{d}
\rho\mathrm{d}r+v(s,X^{\pi}_s)\right],
$$
and the related Hamilton-Jacobi-Bellman (HJB) equation is
\begin{align}
0=&\min_{\pi_t\in\mathcal{A}[t,t]}\bigg[\partial_tv(t,x)+
\int_{\mathbb{R}}\bigg(\hat{f}(x,\rho)+\lambda\ln\pi_t(\rho) \nonumber \\
&+\frac{1}{2}\hat{\sigma}^2(\beta_t,x,\rho)\partial_{xx}v(t,x)
+\hat{b}(\beta_t,x,\rho)\partial_xv(t,x)\bigg)\pi_t(\rho) \mathrm{d}\rho\bigg], \label{hjb-1}
\end{align}
with terminal condition $v(T,x)=\hat{\Phi}(x)$.

It is worth noting that, equation (\ref{hjb-1}) has an optimal control $\pi^*_t(\cdot)$ which satisfies
\begin{equation}\label{optim-1}
\pi^*_t(\rho)=\frac{\exp\left(-\frac{1}{\lambda}\left(\hat{f}(x,\rho)
+\frac{1}{2}\hat{\sigma}^2(\beta_t,x,\rho)\partial_{xx}v(t,x)
+\hat{b}(\beta_t,x,\rho)\partial_xv(t,x) \right) \right)}{\int_{\mathbb{R}}\exp\left(-\frac{1}{\lambda}\left(\hat{f}(x,u)
+\frac{1}{2}\hat{\sigma}^2(\beta_t,x,u)\partial_{xx}v(t,x)
+\hat{b}(\beta_t,x,u)\partial_xv(t,x) \right) \right)\mathrm{d}u}.
\end{equation}
Letting $\lambda\to 0$, $\pi^*_t(\cdot)$  reduces to a Dirac measure, where
$$
\pi^*_t(\rho)=\delta_{\rho^*_t}(\rho),\quad \rho\in \mathbb{R},
$$
and $\rho_t^*$ solves the following equation of $\rho$,
\begin{equation}\label{optim-2}
0=\partial_{\rho}\hat{f}(x,\rho)+\hat{\sigma}(\beta_t,x,\rho)
\partial_{\rho}\hat{\sigma}(\beta_t,x,\rho)\partial_{xx}v(t,x)
+\partial_{\rho}\hat{b}(\beta_t,x,\rho)\partial_xv(t,x).
\end{equation}
By Theorem \ref{the-1}, we have that equation (\ref{optim-2}) has at least one solution $\rho_t^*=\beta_t$, and thus
$$
0=\partial_{\rho}\hat{f}(x,\beta_t)+\hat{\sigma}(\beta_t,x,\beta_t)
\partial_{\rho}\hat{\sigma}(\beta_t,x,\beta_t)\partial_{xx}v(t,x)
+\partial_{\rho}\hat{b}(\beta_t,x,\beta_t)\partial_xv(t,x).
$$
Furthermore, the function of $\rho$, $\mathcal{L}(t,x,\beta_t,\rho)$ takes the minimum value at $\beta_t$, where
$$
\mathcal{L}(t,x,\beta_t,\rho)=\hat{f}(x,\rho)
+\frac{1}{2}\hat{\sigma}^2(\beta_t,x,\rho)\partial_{xx}v(t,x)
+\hat{b}(\beta_t,x,\rho)\partial_xv(t,x).
$$
Since $\lambda>0$, thus $\pi^*_t(\rho)$ takes the maximum value at $\beta_t$. The above observations motive us to find the optimal estimation for the parameter $\beta(\cdot)$ from the optimal density function $\pi^*(\cdot)$.

Based on the above analysis, we conclude the following results.
\begin{theorem}\label{the-2}
If $\mathcal{L}(t,x,\beta_t,\rho)$ satisfies for any $\rho\neq \beta_t$,
\begin{equation}\label{cond-1}
\mathcal{L}(t,x,\beta_t,\rho)>\mathcal{L}(t,x,\beta_t,\beta_t).
\end{equation}
We have that
$$
 {\arg\max}_{\rho\in \mathbb{R}}\pi^*_t(\rho)=\beta_t,\ \ 0\leq t\leq T.
$$
\end{theorem}
\begin{proof}
From condition (\ref{cond-1}), we have that $\mathcal{L}(t,x,\beta_t,\rho)$ takes the unique minimum value at $\beta_t$, and $\pi^*_t(\rho)$ takes the unique maximum value at $\beta_t$. Obviously, $ {\arg\max}_{\rho\in \mathbb{R}}\pi^*_t(\rho)$ is equal to the parameter $\beta(\cdot)$.
\end{proof}

In the following section, we consider a linear stochastic differential equation (SDE) model to verify the results given in Theorem \ref{the-2}, where the drift or diffusion term contains unknown parameter.
\section{Linear SDE problem}\label{sec:4}

We now consider the problem of the Linear SDE case, where the drift or diffusion term is allowed to contain the unknown parameter. We first study the diffusion term with unknown parameter case, and then investigate the drift term with unknown parameter.
\subsection{The diffusion term with unknown parameter}
We consider the following linear SDE,
\begin{equation}\label{exam-sde-1}
\mathrm{d}X^{u}_s=\left[X^{u}_s+u_s\right]\mathrm{d}s
+\left[\beta_sX^{u}_s+u_s\right]\mathrm{d}W_s,
\end{equation}
with the initial condition $X^u_t=x$ and parameter $\beta(\cdot)$ needs to be estimated. The cost functional is given as follows,
\begin{equation}\label{exam-cost-1}
J(t,x;u(\cdot))=\mathbb{E}[\left(X^u_T\right)^2],
\end{equation}
where $u(\cdot)\in \mathcal{U}[t,T]$. Based on the classical optimal control theory, the related HJB equation is given by
$$
0=\partial_tv(t,x)+\inf_{u\in U}\left[\frac{1}{2}(\beta_tx+u)^2\partial_{xx}v(t,x)+
(x+u)\partial_xv(t,x)\right].
$$
The optimal control is given by
$$
u^*(s,X^*_s)=-(1+\beta_s)X^{*}_s,\quad t\leq s\leq T.
$$
We take the parameter $\beta_sX^{*}_s$ as $\rho_s$, and obtain
$$
u^{\rho}(s,X^{\rho}_s)=-X^{\rho}_s-\rho_s,\quad t\leq s\leq T.
$$
\begin{remark}
In this present paper, we aim to establish the theory to estimate the parameter appeared in the state's dynamic equation. To derive the explicit solution of the optimal control $\pi^*(\cdot)$, here, we take place the parameter $\beta(\cdot)X^{*}(\cdot)$ as $\rho(\cdot)$ which is essentially same with the method developed in Section \ref{sec:3}. Note that, when we obtain the optimal control $\rho(\cdot)$, we can divide $\rho^*(\cdot)$ by the state $x$. Then we can obtain the optimal estimation for the parameter $\beta(\cdot)$. Indeed, when deal with some special problems, this kind of structure method should be considered according the properties of the problem.
\end{remark}

Then equation (\ref{eq-sde-1}) becomes
\begin{equation}\label{exam-sde-2}
\mathrm{d}X^{\rho}_s=-\rho_s\mathrm{d}s
+(\beta_s-1-\rho_s{/}X^{\rho}_s)X^{\rho}_s\mathrm{d}W_s,\quad X^{\rho}_t=x.
\end{equation}
Now, applying Theorem \ref{the-rep}, we can formulate the RL stochastic optimal control problem. The exploratory dynamic equation is,
\begin{equation}\label{exam-sde-3}
\mathrm{d}X^{\pi}_s=\tilde{b}(\beta_s,X_s^{\pi},\pi_s)\mathrm{d}s
+\tilde{\sigma}(\beta_s,X_s^{\pi},{\pi}_s)\mathrm{d}W_s,\ X^{\pi}_t=x,
\end{equation}
where
$$
\tilde{b}(\beta_s,x,\pi_s)=-\int_{\mathbb{R}}\rho\pi_s(\rho)\mathrm{d}\rho,
$$
and
$$
\tilde{\sigma}(\beta_s,x,{\pi}_s)=
\sqrt{\int_{\mathbb{R}}(\beta_s-1-\rho{/}X^{\rho}_s)^2(X^{\rho}_s)^2\pi_s(\rho)\mathrm{d}\rho}.
$$
The exploratory cost functional becomes
\begin{equation}\label{exam-cost-3}
J(t,x;\pi(\cdot))=\mathbb{E}[\left(X^{\pi}_T\right)^2].
\end{equation}

By the formula of the optimal control in (\ref{optim-1}), we have that
\begin{equation}\label{exam-optim-1}
\pi^*_t(\rho)=\frac{\exp\left(-\frac{1}{\lambda}\left(\frac{1}{2}(\beta_t-1-\rho{/}x)^2x^2\partial_{xx}v(t,x)
-\rho\partial_xv(t,x) \right) \right)}{\int_{\mathbb{R}}\exp\left(-\frac{1}{\lambda}\left(\frac{1}{2}(\beta_t-1-u{/}x)^2x^2\partial_{xx}v(t,x)
-u\partial_xv(t,x) \right) \right)\mathrm{d}u},
\end{equation}
and thus
$$
\pi^*_t(\rho)=\frac{1}{\sqrt{2\pi \sigma^2(t,x)}}\exp\left(-\frac{\left(\rho-\mu(t,x)\right)^2}{2\sigma^2(t,x)}\right),
$$
where
\begin{align*}
& \mu(t,x)=-\frac{a_2(t,x)}{2a_1(t,x)};\\
& \sigma^2(t,x)=\frac{\lambda}{2a_1(t,x)};\\
& a_1(t,x)=\frac{\partial_{xx}v(t,x)}{2};\\
& a_2(t,x)=(1-\beta_t)x\partial_{xx}v(t,x)-\partial_xv(t,x).
\end{align*}
Putting $\pi^*_t(\rho)$ into equation (\ref{hjb-1}), and by a simple calculation, one obtains
\begin{equation}\label{exam-hjb-1}
0=\partial_tv(t,x)+(\beta_t-1)^2x^2a_1(t,x)-\mu^2(t,x)a_1(t,x)
-\frac{\lambda}{2}\ln\frac{\pi\lambda}{a_1(t,x)}.
\end{equation}

We assume that the classical solution of equation (\ref{exam-hjb-1}) satisfies
$$
v(t,x)=\alpha_1(t)x^2+\alpha_2(t).
$$
Then, it follows that
\begin{align*}
& \mu(t,x)=\beta_tx;\\
& \sigma^2(t,x)=\frac{\lambda}{2\alpha_1(t)};\\
& a_1(t,x)=\alpha_1(t);\\
& a_2(t,x)=-2\beta_t\alpha_1(t)x.
\end{align*}
Now, we put the formula of $v(t,x)$ into equation (\ref{exam-hjb-1}) and obtain that
\begin{equation}\label{exam-hjb-2}
0=\alpha'_1(t)x^2+\alpha'_2(t)+(1-2\beta_t)\alpha_1(t)x^2
-\frac{\lambda}{2}\ln\frac{\pi\lambda}{\alpha_1(t)}.
\end{equation}
Thus, $\alpha_1(t),\ \alpha_2(t)$ satisfy the following equations
\begin{align*}
& 0=\alpha'_1(t)+(1-2\beta_t)\alpha_1(t),\quad \alpha_1(T)=1;\\
& 0=\alpha'_2(t)-\frac{\lambda}{2}\ln\frac{\pi\lambda}{\alpha_1(t)},\quad \alpha_2(T)=0,
\end{align*}
and
\begin{align*}
& \alpha_1(t)=\exp\left(\int_t^T(1-2\beta_s)\mathrm{d}s\right);\\
& \alpha_2(t)=\int_t^T\frac{\lambda}{2}\ln\frac{\pi\lambda}{\alpha_1(s)}\mathrm{d}s,
\end{align*}
which solve the value function $v(t,x)$ with the optimal control $\pi^*(\rho)$ satisfying the Gaussian density function,
$$
\pi_t^*(\rho)=\frac{1}{\sqrt{\pi \lambda{/}\alpha_1(t)}}\exp\left(-\frac{\left(\rho-(\beta_tx)^2\right)^2}
{\lambda{/}\alpha_1(t)}\right),\quad 0\leq t\leq T.
$$
For notation simplicity, we denote by
$$
\pi_t^*(\cdot)\overset{d}{=} \mathcal{N}(\beta_tx,\ \frac{\lambda}{2\alpha_1(t)}),
$$
where
$$
\alpha_1(t)=\exp\left(\int_t^T(1-2\beta_s)\mathrm{d}s\right).
$$

We conclude the above results in the following theorem.
\begin{theorem}\label{the-3}
For the exploratory dynamic equation (\ref{exam-sde-3}) with cost functional (\ref{exam-cost-3}), the optimal control is given as follows,
$$
\pi_t^*(\cdot)\overset{d}{=} \mathcal{N}(\beta_tx,\ \frac{\lambda}{2\alpha_1(t)}),\ 0\leq t\leq T,
$$
where $\displaystyle \mathcal{N}(\beta_tx,\ \frac{\lambda}{2\alpha_1(t)})$ denotes the Gaussian density function with mean $\beta_tx$ and variance $\displaystyle \frac{\lambda}{2\alpha_1(t)}$.
\end{theorem}

\begin{remark}
In Theorem \ref{the-3}, we show the explicit formula of the optimal control $\pi_t^*(\cdot),\ 0\leq t\leq T$. In the control $\pi_t^*(\cdot),\ 0\leq t\leq T$, we can use the mean $\beta_tx$ to estimate the parameter $\beta(\cdot)$. In this exploratory stochastic optimal control problem, the variance $\frac{\lambda}{2\alpha_1(t)}$ of the optimal control $\pi_t^*(\cdot)$ investigates the efficiency of exploratory of the RL. It is worth noting that, let $\lambda\to 0$, $\pi_t^*(\cdot)$ reduces to the Dirac measure $\delta_{\beta_tx}(\cdot)$.

Now, we consider the following classical optimal control problem. The state satisfies
\begin{equation}\label{re-exam-sde-1}
\mathrm{d}X^{\rho}_s=-\rho_s\mathrm{d}s
+\left[(\beta_s-1)X^{\rho}_s-\rho_s\right]\mathrm{d}W_s,,\ X^{\rho}_t=x.
\end{equation}
and the cost functional is
\begin{equation}\label{re-exam-cost-1}
J(t,x;u(\cdot))=\mathbb{E}[\left(X^u_T\right)^2].
\end{equation}
The value function is defined by
$$
\hat{v}(t,x)=\inf_{\rho(\cdot)\in \mathcal{D}[t,T]}J(t,x;u(\cdot)),
$$
and satisfies the following HJB equation,
\begin{equation}\label{re-hjb-1}
0=\min_{\rho\in\mathbb{R}}\left[\partial_t\hat{v}(t,x)+\frac{1}{2}\left[(\beta_t-1)x-\rho\right]^2\partial_{xx}\hat{v}(t,x)
-\rho\partial_x\hat{v}(t,x)\right].
\end{equation}
The optimal control of HJB equation (\ref{re-hjb-1}) is
$$
\rho^*_t=\frac{(\beta_t-1)x\partial_{xx}\hat{v}(t,x)+\partial_x\hat{v}(t,x)}{\partial_{xx}\hat{v}(t,x)}.
$$
Then, putting $\rho_t^*$ into equation (\ref{re-hjb-1}), and we assume the formula of the value function $\hat{v}(t,x)$,
$$
\hat{v}(t,x)=b_1(t)x^2+b_2(t).
$$
Thus, $b_1(t),\ b_2(t)$ satisfy the following equations
\begin{align*}
& 0=b'_1(t)+(1-2\beta_t)b_1(t),\quad b_1(T)=1;\\
& 0=b'_2(t),\quad b_2(T)=0,
\end{align*}
and
\begin{align*}
& b_1(t)=\exp\left(\int_t^T(1-2\beta_s)\mathrm{d}s\right);\\
& b_2(t)=0.
\end{align*}
The value function $\hat{v}(t,x)$ is given as
$$
\hat{v}(t,x)=\exp\left(\int_t^T(1-2\beta_s)\mathrm{d}s\right)x^2,
$$
with the optimal control $\rho^*_t=\beta_tx$.

Thus, the optimal control $\rho^*_t$ of the classical optimal control problem is consistent with the mean of the optimal density function $\pi^*_t$. Therefore, based on the reinforcement learning (RL) stochastic optimal control structure presented in this study, we can use the RL method to learn the optimal density function $\pi^*_t$ and then obtains the estimation for parameter $\beta(\cdot)$.
\end{remark}

\subsection{The drift term with unknown parameter}\label{subsec:4.2}
We could also consider the following model where the state satisfies
\begin{equation}\label{2exam-sde-1}
\mathrm{d}X^{u}_s=\left[\beta_s X^{u}_s+u_s\right]\mathrm{d}s
+\left[X^{u}_s+u_s\right]\mathrm{d}W_s,
\end{equation}
with the initial condition $X^u_t=x$ and parameter $\beta(\cdot)$ needs to be estimated. Note that the parameter $\beta(\cdot)$ appears in the drift term of the state equation (\ref{2exam-sde-1}). We construct the following cost functional differed from (\ref{re-exam-cost-1}),
\begin{equation}\label{2exam-cost-1}
J(t,x;u(\cdot))=\mathbb{E}\int_{t}^T[(u_s+X^u_s+\beta_sX^u_s)^2]\mathrm{d}s.
\end{equation}
\begin{remark}
In the cost functional (\ref{2exam-cost-1}), there is unknown term $\beta_sX^u_s$. In fact, from  equation (\ref{2exam-sde-1}), we can observe the value of $\beta_sX^u_s,\ t\leq s\leq T$. Precisely, Based on an input $u(\cdot)$, we can obtain the value of $J(t,x;u(\cdot))$.
\end{remark}

Indeed, when we consider the cost functional (\ref{re-exam-cost-1}), the optimal control admits the unique optima control $u^*(\cdot)=-2X^*(\cdot)$ which cannot be used to estimate the parameter $\beta(\cdot)$.
From cost functional (\ref{2exam-cost-1}), the related optimal control is given as follows:
$$
u^*_s=-X_s^*-{\beta_sX_s^*},\ t\leq s\leq T.
$$

We take place ${\beta_sX_s^*}$ as $\rho_s$, and thus
$$
u^{\rho}_s=-X_s^{\rho}-{\rho_s},\ t\leq s\leq T.
$$
Then, equation (\ref{2exam-sde-1}) becomes
\begin{equation}\label{2exam-sde-2}
\mathrm{d}X^{\rho}_s=[(\beta_s-1)X_s^{\rho}-\rho_s]\mathrm{d}s
-{\rho}_s\mathrm{d}W_s,\ X^{\rho}_t=x.
\end{equation}
Based on Theorem \ref{the-rep}, the exploratory dynamic equation is,
\begin{equation}\label{2exam-sde-3}
\mathrm{d}X^{\pi}_s=\tilde{b}(\beta_s,X_s^{\pi},\pi_s)\mathrm{d}s
+\tilde{\sigma}(\beta_s,X_s^{\pi},{\pi}_s)\mathrm{d}W_s,\ X^{\pi}_t=x,
\end{equation}
where
$$
\tilde{b}(\beta_s,x,\pi_s)=\int_{\mathbb{R}}[(\beta_s-1)x-\rho]\pi_s(\rho)\mathrm{d}\rho,
$$
and
$$
\tilde{\sigma}(\beta_s,x,{\pi}_s)=
\sqrt{\int_{\mathbb{R}}\rho^2\pi_s(\rho)\mathrm{d}\rho}.
$$
The exploratory cost functional is,
\begin{equation}\label{2exam-cost-2}
J(t,x;\pi(\cdot))=\mathbb{E}\left[\int_{t}^T
\int_{\mathbb{R}}(\rho-\beta_sX^{\pi}_s)^2\pi_s(\rho)\mathrm{d}\rho\mathrm{d}s\right].
\end{equation}

By the formula of the optimal control in (\ref{optim-1}), we have that
\begin{equation}\label{2exam-optim-1}
\pi^*_t(\rho)=\frac{\exp\left(-\frac{1}{\lambda}\left((\rho-\beta_tx)^2
+\frac{1}{2}\rho^2\partial_{xx}v(t,x)
+[(\beta_t-1)x-\rho]\partial_xv(t,x) \right) \right)}{\int_{\mathbb{R}}\exp\left(-\frac{1}{\lambda}\left((u-\beta_tx)^2
+\frac{1}{2}u^2\partial_{xx}v(t,x)
+[(\beta_t-1)x-u]\partial_xv(t,x) \right) \right)\mathrm{d}u},
\end{equation}
which satisfies the Gaussian density function with mean
$$
\mu(t,x)=\frac{2\beta_tx+\partial_xv(t,x)}{2+\partial_{xx}v(t,x)},
$$
and variance
$$
\sigma^2(t,x)=\frac{\lambda}{2+\partial_{xx}v(t,x)}.
$$
We assume
$$
v(t,x)=\alpha_1(t)x^2+\alpha_2(t).
$$
The related HJB equation becomes
\begin{equation}\label{2exam-hjb-1}
0=\alpha'_1(t)x^2+\alpha'_2(t)+\beta_t^2x^2+2(\beta_t-1)\alpha_1(t)x^2
-\mu^2(t,x)(1+\alpha_1(t))-\frac{\lambda}{2}\ln\frac{\pi\lambda}{1+\alpha_1(t)},
\end{equation}
and derives that
\begin{align*}
& 0=\alpha'_1(t)(1+\alpha_1(t))+\beta_t^2\alpha_1(t)+2\beta_t\alpha^2_1(t)-2\alpha_1(t)-3\alpha_1^2(t);\\
& 0=\alpha_2(t)-\frac{\lambda}{2}\ln\frac{\pi\lambda}{1+\alpha_1(t)},
\end{align*}
with $\alpha_1(T)=\alpha_2(T)=0$, and thus
\begin{align*}
& \alpha_1(t)=0;\\
& \alpha_2(t)=\frac{\lambda(T-t)}{2}\ln(\pi \lambda),
\end{align*}
which solves the value function $v(t,x)$ with the optimal control $\pi^*(\rho)$,
$$
\pi_t^*(\rho)=\frac{1}{\sqrt{\pi \lambda}}\exp\left(-\frac{\left(\rho-\beta_tx\right)^2}
{\lambda}\right),\quad 0\leq t\leq T.
$$
For notation simplicity, we denote by
$$
\pi_t^*(\cdot)\overset{d}{=} \mathcal{N}(\beta_tx,\ \frac{\lambda}{2}).
$$
Based on RL, we can learn the mean $\beta_tx$ from the observation data. Furthermore, we can use the variance $\displaystyle \frac{\lambda}{2}$ to adjust the exploratory rate.

In this section, we have considered the parameters appeared in drift or diffusion term. In the following Section \ref{sec:5}, we aim to consider a more general case where both drift and diffusion terms contain parameter.

\section{Extension}\label{sec:5}
Now, we general the model in Section \ref{sec:4} to the case where both the drift and diffusion terms contain unknown parameters. We show the details how to estimate the parameters in the drift and diffusion terms. Then, we give an example to verify the main results.

\subsection{General model}
We consider the following stochastic differential equation with the deterministic parameters $\alpha_s,\beta_s\in \mathbb{R},\ t\leq s\leq T$, and control $u(\cdot)\in \mathcal{U}[t,T]$,
\begin{equation}\label{ge-eq-sde}
\mathrm{d}X^u_s=b(\alpha_s,X_s^u,u_s)\mathrm{d}s+\sigma(\beta_s,X_s^u,u_s)\mathrm{d}W_s,\ X^u_t=x.
\end{equation}
We construct two cost functionals,
\begin{equation}\label{ge-cost-1}
J_1(t,x;u(\cdot))=\mathbb{E}\left[\int_t^Tf_1(X_s^u,u_s)\mathrm{d}s+\Phi_1(X_T^u)\right],
\end{equation}
and
\begin{equation}\label{ge-cost-2}
J_2(t,x;u(\cdot))=\mathbb{E}\left[\int_t^Tf_2(X_s^u,u_s)\mathrm{d}s+\Phi_2(X_T^u)\right],
\end{equation}
to estimate the parameters $\alpha(\cdot)$ and $\beta(\cdot)$, respectively.

\textbf{Firstly}, we show how to estimate $\alpha(\cdot)$ from cost functional $J_1(t,x;u(\cdot))$. By choosing appropriate functions $f_1(\cdot)$ and $\Phi_1(\cdot)$, we obtain a feedback optimal control $u^{1,*}_s=u^{1,*}(\alpha_s,s,X^{1,*}_s),\ t\leq s\leq T$  from cost functional $J_1(t,x;u(\cdot))$, where $X^{1,*}(\cdot)$ is the optimal state with the optimal control $u^{1,*}(\cdot)$. However, we need that $u^{1,*}(\cdot)$ contains the parameter $\alpha(\cdot)$ only. The question is that if $u^{1,*}(\cdot)$ contains the parameter $\beta(\cdot)$, we don't know the value of the new input control (see $u^{\rho}(\cdot)$). We take place $\alpha(\cdot)$ as a deterministic process $\rho(\cdot)$ in the feedback optimal control $u^{1,*}(\alpha_s,s,X^{1,*}_s)$, and denotes by $u^{\rho}_s=u^{1,*}(\rho_s,s,X^{\rho}_s)$, where $X^{\rho}_s$ satisfies the following stochastic differential equation,
\begin{equation}\label{ge-eq-sde-1}
\mathrm{d}X^{\rho}_s=\hat{b}(\alpha_s,X_s^{\rho},\rho_s)\mathrm{d}s
+\hat{\sigma}(\beta_s,X_s^{\rho},{\rho}_s)\mathrm{d}W_s,\ X^{\rho}_t=x,
\end{equation}
where $\hat{b}(\alpha_s,X_s^{\rho},\rho_s)={b}(\alpha_s,X_s^{\rho},u^{\rho}_s)$ and $\hat{\sigma}(\beta_s,X_s^{\rho},\rho_s)={\sigma}(\beta_s,X_s^{\rho},u^{\rho}_s)$. Furthermore, we rewrite the cost functional (\ref{ge-cost-1}) as follows,
\begin{equation}\label{ge-cost-3}
J_1(t,x;\rho(\cdot))=\mathbb{E}\left[\int_t^Tf_1(X_s^{\rho},\rho_s)\mathrm{d}s+\Phi_1(X_T^{\rho})\right],
\end{equation}
where $\rho_s\in \mathbb{R},\ t\leq s\leq T$. Based on Theorem \ref{the-rep}, we consider the following exploratory dynamic equation,
\begin{equation}\label{ge-eq-sde-2}
\mathrm{d}X^{\pi}_s=\tilde{b}(\alpha_s,X_s^{\pi},\pi_s)\mathrm{d}s
+\tilde{\sigma}(\beta_s,X_s^{\pi},{\pi}_s)\mathrm{d}W_s,\ X^{\pi}_t=x,
\end{equation}
where $\pi_s$ is the density function of the input control $\rho_s$,
$$
\tilde{b}(\alpha_s,x,\pi_s)=\int_{\mathbb{R}}\hat{b}(\alpha_s,x,\rho)\pi_s(\rho)\mathrm{d}\rho,
$$
and
$$
\tilde{\sigma}(\beta_s,x,{\pi}_s)=
\sqrt{\int_{\mathbb{R}}\hat{\sigma}^2(\beta_s,x,\rho)\pi_s(\rho)\mathrm{d}\rho}
$$
for $t\leq s\leq T$. The cost functional (\ref{ge-cost-3}) is rewritten as follows,
\begin{equation}\label{ge-cost-4}
\tilde{J}_1(t,x;\pi(\cdot))=
\mathbb{E}\left[\int_t^T\int_{\mathbb{R}}\left[f_1(X_s^{\pi},\rho)\pi_s(\rho)
+\lambda\pi_s(\rho)\ln\pi_s(\rho)\right]\mathrm{d}
\rho\mathrm{d}s+\Phi_1(X_T^{\pi})\right].
\end{equation}

\textbf{Secondly}, we estimate $\beta(\cdot)$ by cost functional $J_2(t,x;u(\cdot))$. We choose appropriate functions $f_2(\cdot)$ and $\Phi_2(\cdot)$, such that the feedback optimal control $u^{2,*}_s=u^{2,*}(\beta_s,s,X^{2,*}_s),\ t\leq s\leq T$ of the cost functional $J_2(t,x;u(\cdot))$ contains the parameter $\beta(\cdot)$, where $X^{2,*}(\cdot)$ is the optimal state with the optimal control $u^{2,*}(\cdot)$. We take place $\beta(\cdot)$ as a deterministic process $\gamma(\cdot)$ in the feedback optimal control $u^{2,*}(\beta_s,s,X^{2,*}_s)$, and denotes by $u^{\gamma}_s=u^{2,*}(\gamma_s,s,X^{\gamma}_s)$, where $X^{\gamma}_s$ satisfies the following stochastic differential equation.
\begin{equation}\label{22ge-eq-sde-1}
\mathrm{d}X^{\gamma}_s=\hat{b}(\alpha_s,X_s^{\gamma},\gamma_s)\mathrm{d}s
+\hat{\sigma}(\beta_s,X_s^{\gamma},{\gamma}_s)\mathrm{d}W_s,\ X^{\gamma}_t=x.
\end{equation}
where $\hat{b}(\alpha_s,X_s^{\gamma},\gamma_s)={b}(\alpha_s,X_s^{\gamma},u^{\gamma}_s)$ and $\hat{\sigma}(\beta_s,X_s^{\gamma},\gamma_s)={\sigma}(\beta_s,X_s^{\gamma},u^{\gamma}_s)$. Furthermore, we rewrite the cost functional (\ref{ge-cost-2}) as follows,
\begin{equation}\label{22ge-cost-2}
J_2(t,x;\gamma(\cdot))=\mathbb{E}\left[\int_t^Tf_2(X_s^{\gamma},\gamma_s)\mathrm{d}s
+\Phi_2(X_T^{\gamma})\right],
\end{equation}
where $\gamma_s\in \mathbb{R},\ t\leq s\leq T$. Then, based on Theorem \ref{the-rep}, we can introduce the exploratory dynamic equations almost same with (\ref{ge-eq-sde-2}) and (\ref{ge-cost-4}). For notations simplicity, we omit them.

\subsection{Linear SDE with unknown parameters}
Now, we consider the following example.
\begin{equation}\label{ge-exam-sde-1}
\mathrm{d}X^{u}_s=\left[\alpha_s X^{u}_s+u_s\right]\mathrm{d}s
+\left[\beta_sX^{u}_s+u_s\right]\mathrm{d}W_s,
\end{equation}
where $\alpha(\cdot)$ and $\beta(\cdot)$ need to be estimated. Note that, equation (\ref{ge-exam-sde-1}) contains the parameters $\alpha(\cdot)$ and $\beta(\cdot)$. Thus, we cannot use the cost functional developed in Subsection \ref{subsec:4.2} to investigate the estimation for $\alpha(\cdot)$, directly. Therefore, we first consider to estimate the parameter $\beta(\cdot)$. Based on the estimation of $\beta(\cdot)$, we can use the cost functional constructed in Subsection \ref{subsec:4.2} to estimate the parameter $\alpha(\cdot)$.

\textbf{Step 1:} We first construct the cost functional $J_1(t,x;u(\cdot))$ used to estimate parameter $\beta(\cdot)$, where
\begin{equation}\label{2ge-exam-cost-1}
J_1(t,x;u(\cdot))=\mathbb{E}[(X^{u}_T)^2].
\end{equation}
Based on the classical optimal control theory, the related HJB equation is given by
$$
0=\partial_tv(t,x)+\inf_{u\in U}\left[\frac{1}{2}(\beta_tx+u)^2\partial_{xx}v(t,x)+
(\alpha_tx+u)\partial_xv(t,x)\right].
$$
The related optimal control is
$$
u^*(s,X^*_s)=-(1+\beta_s)X^{*}_s,\quad t\leq s.
$$
We take place ${\beta_sX_s^*}$ as $\gamma_s$, and have that
$$
u^{\gamma}_s=-{\gamma_s}-{X_s^*},\ t\leq s\leq T.
$$
Equation (\ref{ge-exam-sde-1}) becomes
\begin{equation}\label{2ge-exam-sde-2}
\mathrm{d}X^{\gamma}_s=[(\alpha_s-1)X^{\gamma}_s-\gamma_s]\mathrm{d}s
+[(\beta_s-1)X^{\gamma}_s-{\gamma}_s]\mathrm{d}W_s,\ X^{\gamma}_t=x.
\end{equation}
From Theorem \ref{the-rep}, the exploratory dynamic equation is,
\begin{equation}\label{2ge-exam-sde-3}
\mathrm{d}X^{\pi}_s=\tilde{b}(\alpha_s,X_s^{\pi},\pi_s)\mathrm{d}s
+\tilde{\sigma}(\beta_s-\alpha_s,X_s^{\pi},{\pi}_s)\mathrm{d}W_s,\ X^{\pi}_t=x,
\end{equation}
where
$$
\tilde{b}(\beta_s,x,\pi_s)=\int_{\mathbb{R}}[(\alpha_s-1)x-\gamma]\pi_s(\gamma)\mathrm{d}\gamma,
$$
and
$$
\tilde{\sigma}(\beta_s,x,{\pi}_s)=
\sqrt{\int_{\mathbb{R}}[(\beta_s-1)x-\gamma]^2\pi_s(\rho)\mathrm{d}\rho}.
$$
The exploratory cost functional is,
\begin{equation}\label{2ge-exam-cost-2}
J_1(t,x;\pi(\cdot))=\mathbb{E}[(X^{\pi}_T)^2].
\end{equation}

By the formula of the optimal control in (\ref{optim-1}), we have that
\begin{equation}\label{2ge-exam-optim-1}
\pi^{1,*}_t(\rho)=\frac{\exp\left(-\frac{1}{\lambda}\left(
\frac{1}{2}[(\beta_t-1)x-\gamma]^2\partial_{xx}v(t,x)
+[(\alpha_t-1)x-\gamma]\partial_xv(t,x) \right) \right)}{\int_{\mathbb{R}}\exp\left(-\frac{1}{\lambda}
\left(\frac{1}{2}[(\beta_t-1)x-u]^2\partial_{xx}v(t,x)
+[(\alpha_t-1)x-u]\partial_xv(t,x) \right) \right)\mathrm{d}u},
\end{equation}
which satisfies the Gaussian density function with mean
$$
\mu(t,x)=\frac{(\beta_t-1)x\partial_{xx}v+\partial_xv(t,x)}{\partial_{xx}v(t,x)},
$$
and variance
$$
\sigma^2(t,x)=\frac{\lambda}{\partial_{xx}v(t,x)}.
$$
We assume
$$
v(t,x)=\theta_1(t)x^2+\theta_2(t).
$$
The related HJB equation becomes
\begin{equation}\label{2ge-exam-hjb-1}
0=\theta'_1(t)x^2+\theta'_2(t)+(\beta_t-1)^2\theta_1(t)x^2+2(\alpha_t-1)\theta_1(t)x^2
-\mu^2(t,x)\theta_1(t)-\frac{\lambda}{2}\ln\frac{\pi\lambda}{\theta_1(t)},
\end{equation}
and derives that
\begin{align*}
& \theta_1(t)=\exp\left(\int_t^T[2(\alpha_s-\beta_s)-1]\mathrm{d}s\right);\\
& \theta_2(t)=\frac{\lambda}{2}\int_t^T\ln(\frac{\pi\lambda}{\theta_1(s)})\mathrm{d}s,
\end{align*}
which solves the value function $v(t,x)$ with the optimal control $\pi^{1,*}(\gamma)$ satisfies Gaussian density function,
$$
\pi_t^{1,*}(\gamma)=\sqrt{\frac{\theta_1(t)}{\pi \lambda}}\exp\left(-\frac{\left(\gamma-\beta_tx\right)^2}
{\lambda{/}\theta_1(t)}\right),\quad 0\leq t\leq T.
$$
For notation simplicity, we denote by
$$
\pi_t^{1,*}(\cdot)\overset{d}{=} \mathcal{N}(\beta_tx,\ \frac{\lambda}{2\theta_1(t)}),
$$
where
$$
\theta_1(t)=\exp\left(\int_t^T[2(\alpha_s-\beta_s)-1]\mathrm{d}s\right).
$$

\textbf{Step 2:} We then construct the cost functional $J_2(t,x;u(\cdot))$ used to estimate parameter $\alpha(\cdot)$, where
\begin{equation}\label{ge-exam-cost-1}
J_2(t,x;u(\cdot))=\mathbb{E}\int_{t}^T[(u_s+\alpha_sX^u_s+\beta_sX^u_s)^2]\mathrm{d}s.
\end{equation}
Minimizing the cost functional (\ref{ge-exam-cost-1}), the related optimal control is given by
$$
u^*_s=-\alpha_sX_s^*-{\beta_sX_s^*},\ t\leq s\leq T.
$$
We take place ${\alpha_sX_s^*}$ as $\rho_s$, and thus
$$
u^{\rho}_s=-{\rho_s}-{\beta_sX_s^*},\ t\leq s\leq T.
$$
Then, equation (\ref{ge-exam-sde-1}) becomes
\begin{equation}\label{ge-exam-sde-2}
\mathrm{d}X^{\rho}_s=[(\alpha_s-\beta_s)X_s^{\rho}-\rho_s]\mathrm{d}s
-{\rho}_s\mathrm{d}W_s,\ X^{\rho}_t=x.
\end{equation}
From Theorem \ref{the-rep}, the exploratory dynamic equation is,
\begin{equation}\label{ge-exam-sde-3}
\mathrm{d}X^{\pi}_s=\tilde{b}(\alpha_s-\beta_s,X_s^{\pi},\pi_s)\mathrm{d}s
+\tilde{\sigma}(\beta_s,X_s^{\pi},{\pi}_s)\mathrm{d}W_s,\ X^{\pi}_t=x,
\end{equation}
where
$$
\tilde{b}(\beta_s,x,\pi_s)=\int_{\mathbb{R}}[(\alpha_s-\beta_s)x-\rho]\pi_s(\rho)\mathrm{d}\rho,
$$
and
$$
\tilde{\sigma}(\beta_s,x,{\pi}_s)=
\sqrt{\int_{\mathbb{R}}\rho^2\pi_s(\rho)\mathrm{d}\rho}.
$$
The exploratory cost functional is,
\begin{equation}\label{ge-exam-cost-2}
J_2(t,x;\pi(\cdot))=\mathbb{E}[\int_{t}^T
\int_{\mathbb{R}}(\rho-\alpha_sX^{\pi}_s)^2\pi_s(\rho)\mathrm{d}\rho\mathrm{d}s].
\end{equation}

By the formula of the optimal control in (\ref{optim-1}), we have that
\begin{equation}\label{ge-exam-optim-1}
\pi^{2,*}_t(\rho)=\frac{\exp\left(-\frac{1}{\lambda}\left((\rho-\alpha_tx)^2
+\frac{1}{2}\rho^2\partial_{xx}v(t,x)
+[(\alpha_t-\beta_t)x-\rho]\partial_xv(t,x) \right) \right)}{\int_{\mathbb{R}}\exp\left(-\frac{1}{\lambda}\left((u-\alpha_tx)^2
+\frac{1}{2}u^2\partial_{xx}v(t,x)
+[(\alpha_t-\beta_t)x-u]\partial_xv(t,x) \right) \right)\mathrm{d}u},
\end{equation}
which satisfies the Gaussian density function with mean
$$
\mu(t,x)=\frac{2\alpha_tx+\partial_xv(t,x)}{2+\partial_{xx}v(t,x)},
$$
and variance
$$
\sigma^2(t,x)=\frac{\lambda}{2+\partial_{xx}v(t,x)}.
$$
We assume
$$
v(t,x)=\theta_1(t)x^2+\theta_2(t).
$$
The related HJB equation becomes
\begin{equation}\label{ge-exam-hjb-1}
0=\theta'_1(t)x^2+\theta'_2(t)+\alpha_t^2x^2+2(\alpha_t-\beta_t)\theta_1(t)x^2
-\mu^2(t,x)(1+\theta_1(t))-\frac{\lambda}{2}\ln\frac{\pi\lambda}{1+\theta_1(t)}.
\end{equation}
and derives that
\begin{align*}
& \theta_1(t)=0;\\
& \theta_2(t)=\frac{\lambda(T-t)}{2}\ln(\pi \lambda),
\end{align*}
which solves the value function $v(t,x)$ with the optimal control $\pi^{2,*}(\rho)$ satisfies Gaussian density function,
$$
\pi_t^{2,*}(\rho)=\frac{1}{\sqrt{\pi \lambda}}\exp\left(-\frac{\left(\rho-\alpha_tx\right)^2}
{\lambda}\right),\quad 0\leq t\leq T.
$$
For notation simplicity, we denote by
$$
\pi_t^{2,*}(\cdot)\overset{d}{=} \mathcal{N}(\alpha_tx,\ \frac{\lambda}{2}).
$$

We conclude the above main results in the following theorem.
\begin{theorem}\label{the-4}
When both the drift and diffusion terms of the linear SDE (\ref{ge-eq-sde}) contain the unknown parameters $\alpha(\cdot)$ and $\beta(\cdot)$, based on the exploratory equations (\ref{2ge-exam-sde-3}) and (\ref{ge-exam-sde-3}), and related cost functionals  (\ref{2ge-exam-cost-2}) and (\ref{ge-exam-cost-2}), we have the optimal density functions for the parameters $\beta(\cdot)$ and $\alpha(\cdot)$, respectively,
$$
\pi_t^{1,*}(\cdot)\overset{d}{=} \mathcal{N}(\beta_tx,\ \frac{\lambda}{2\theta_1(t)}),\quad
\pi_t^{2,*}(\cdot)\overset{d}{=} \mathcal{N}(\alpha_tx,\ \frac{\lambda}{2}),
$$
where
$$
 \theta_1(t)=\exp\left(\int_t^T[2(\alpha_s-\beta_s)-1]\mathrm{d}s\right).
$$
Based on $\pi_t^{1,*}(\cdot)$ and $\pi_t^{2,*}(\cdot)$, we can learn the unknown parameters $\beta(\cdot)$ and $\alpha(\cdot)$.
\end{theorem}

\begin{remark}
\cite{WZZ20b} investigated a continuous time mean-variance portfolio selection with reinforcement learning model, and developed an implementable reinforcement learning algorithm. Based on the results in \cite{WZZ20b}, we can obtain the parameters appeared in density functions $\pi_t^{1,*}(\cdot)$ and $\pi_t^{2,*}(\cdot)$, which can be used to estimate parameters $\beta(\cdot)$ and $\alpha(\cdot)$. Based on the estimations of $\beta(\cdot)$ and $\alpha(\cdot)$, we can do empirical analysis for the classical optimal control problem and so on.
\end{remark}

\section{Conclusion}\label{sec:6}

Combining the stochastic optimal control model and the reinforcement learning structure, we develop a novel approach to learn the unknown parameters appeared in the drift and diffusion terms of the stochastic differential equation. By choosing an appropriate cost functional, we can obtain a feedback optimal control $u^*(\cdot)$ which contains the unknown parameter $\beta(\cdot)$, denote by $u^*_s=u^*(\beta_s,s,X^*_s),\ t\leq s\leq T$. We take place the unknown parameter $\beta(\cdot)$ in the optimal control $u^*(\cdot)$ as a new deterministic control $\rho(\cdot)$, and obtain a new input control $u^{\rho}_s=u^{\rho}(\rho_s,s,X^{\rho}_s),\ t\leq s\leq T$. Putting the new control $u^{\rho}(\cdot)$ into the related stochastic differential equation and cost functional. We establish the mathematical framework of the exploratory dynamic equation and cost functional, which is consistent with the structure introduced in \cite{WZZ20a}.

Indeed, the optimal control of the new exploratory control problem can be used to estimate unknown parameters. Therefore, we consider the linear stochastic differential equation case where the drift or diffusion term with unknown parameter. Then, we investigate the general case where both the drift and diffusion terms contain unknown parameters. For the above cases, we show that the optimal density function satisfies a Gaussian distribution which can be used to estimate the unknown parameters. When both the drift and diffusion terms of the linear SDE contain the unknown parameters $\alpha(\cdot)$ and $\beta(\cdot)$, based on the exploratory equations, and related cost functionals, we obtain the optimal density functions for the parameters $\alpha(\cdot)$ and $\beta(\cdot)$, respectively. The optimal density functions are given by
$$
\pi_t^{1,*}(\cdot)\overset{d}{=} \mathcal{N}(\beta_tx,\ \frac{\lambda}{2\theta_1(t)}),\quad
\pi_t^{2,*}(\cdot)\overset{d}{=} \mathcal{N}(\alpha_tx,\ \frac{\lambda}{2}),
$$
where
$$
 \theta_1(t)=\exp\left(\int_t^T[2(\alpha_s-\beta_s)-1]\mathrm{d}s\right).
$$
Based on $\pi_t^{1,*}(\cdot)$ and $\pi_t^{2,*}(\cdot)$, we can learn the unknown parameters $\beta(\cdot)$ and $\alpha(\cdot)$.

In this present paper, we focus on develop the theory of a stochastic optimal control approach with reinforcement learning structure to estimate the unknown parameters appeared in the model. Indeed, based on the existing optimal control learning methods, for example policy evaluation and policy improvement developed in \cite{WZZ20b}, we can learn the optimal density functions, and then obtain the estimations of the unknown parameters. These estimations are useful in the related investment portfolio problems, for example mean-variance investment portfolio. We will further consider these problems in the future work.

\bibliography{gexp}
\end{document}